\documentclass{amsart}
\usepackage[utf8]{inputenc}

\usepackage{amssymb,bbm}

\usepackage[backend=bibtex,style=alphabetic,maxbibnames=15,maxcitenames=6,dateabbrev=false]{biblatex}

\bibliography{LogicBiblio,HamkinsBiblio}

\newtheorem{theorem}{Theorem}

\newtheorem{lemma}[theorem]{Lemma}

\newtheorem{question}[theorem]{Question}

\newtheorem{observation}[theorem]{Observation}

\newtheorem{conjecture}[theorem]{Conjecture}

\newtheorem{definition}[theorem]{Definition}

\newcommand{\one}{{\mathbbm 1}}

\renewcommand{\P}{{\mathbb P}}
\newcommand{\Q}{{\mathbb Q}}

\newcommand{\R}{{\mathbb R}}

\newcommand{\Vbar}{{\overline{V}}}

\newcommand{\from}{\mathbin{\vbox{\baselineskip=2pt\lineskiplimit=0pt
                         \hbox{.}\hbox{.}\hbox{.}}}}

\newcommand{\of}{\subseteq}

\newcommand{\dom}{\mathop{\rm dom}}

\newcommand{\plus}{{+}}

\newcommand{\restrict}{\upharpoonright} 

\newcommand{\forces}{\Vdash}

\newcommand{\Card}[1]{{\left|#1\right|}}

\def\<#1>{\langle#1\rangle}

\newcommand{\ORD}{\mathop{{\rm Ord}}}

\newcommand{\REG}{\mathop{{\rm REG}}}

\newcommand{\ZFC}{{\rm ZFC}}

\newcommand{\GCH}{{\rm GCH}}

\newcommand{\HOD}{{\rm HOD}}

\newcommand{\Add}{{\rm Add}}
\newcommand{\len}{{\rm len}}
\newcommand{\term}{{\rm term}}

\newcommand{\scrA}{{\mathcal A}}

\title{Cohen Forcing and Inner Models}
\author{Jonas Reitz}
 \address[Jonas Reitz]{Department of Mathematics, New York City College of Technology, CUNY, 300 Jay Street, Brooklyn, NY 11201}
 \email{jreitz@citytech.cuny.edu}

\begin{document}

\begin{abstract}
Given an inner model $W \subset V$ and a regular cardinal $\kappa$,  we consider two  alternatives for adding a subset to $\kappa$ by forcing: the Cohen poset $\Add(\kappa,1)$, and the Cohen poset of the inner model $\Add(\kappa,1)^W$. The forcing from $W$ will be at least as strong as the forcing from $V$ (in the sense that forcing with the former adds a generic for the latter) if and only if the two posets have the same cardinality.  On the other hand, a sufficient condition is established for the poset from $V$ to fail to be as strong as that from $W$.  The results are generalized to  $\Add(\kappa,\lambda)$, and to iterations of Cohen forcing where the poset at each stage comes from an arbitrary intermediate inner model.
\end{abstract}

\maketitle

\section{Cohen Forcing and Inner Models}

The most well-known method for adding a subset to a regular cardinal $\kappa$ over the universe $V$ is the Cohen partial order 
$\Add(\kappa,1)$, whose conditions consist of binary sequences bounded in 
$\kappa$ and ordered by end extension. Presented with an inner model 
$W$, however, we can consider an alternative: add a subset to $\kappa$ over the universe $V$ using $\Add(\kappa,1)^W$, the Cohen partial order as defined in $W$.  As $W$ may have fewer bounded subsets of kappa than $V$, these posets need not be equal. How do they compare? 

This situation is not unusual in set theory – it arises, for example, in both of the canonical methods for adding subsets to multiple cardinals (in product forcing we always use the poset of the ground model, whereas in iterations we often use the poset of the extension), and in many other inner and outer model constructions. 

In this paper I analyze $\Add(\kappa,1)^W$ and $\Add(\kappa,1)^V$ with regards to their relative forcing strength, driven by the questions:
\begin{question}\label{Question.DoesCohenforcingfromWaddagenericforV}
Does forcing with $\Add(\kappa,1)^W$ add a generic for $\Add(\kappa,1)^V$?
\end{question}
\begin{question}\label{Question.DoesCohenforcingfromVaddagenericforW}Does forcing with $\Add(\kappa,1)^V$ add a generic for $\Add(\kappa,1)^W$?
\end{question}

Surprisingly, it is the poset from the inner model that exhibits the greater strength.  I will offer a complete characterization of when the first question has a positive answer, and establish a sufficient condition for the second question to have a negative answer.  I will also explore generalizations to $\Add(\kappa,\lambda)$ and iterations, and discuss open questions and applications.

Let's begin by recalling the basic facts about Cohen forcing.  There are a number of equivalent formulations -- those given below were selected for clarity in the subsequent proofs.

\begin{definition}[Cohen forcing]For a regular cardinal $\kappa$, Cohen forcing at $\kappa$ is:
$$\Add(\kappa,1) = \left\{p:\alpha\to 2 \mid \alpha<\kappa \right\} = {}^{<\kappa} 2$$
with ordering by extension,  $p\leq q$ if and only if $q\subseteq p$.

More generally, if $\lambda$ is any ordinal, the Cohen forcing to add $\lambda$-many subsets to $\kappa$ is:
$$\Add(\kappa,\lambda) = \left\{p\from \lambda \to {}^{<\kappa} 2 \; \mid \; |p|<\kappa \right\}.$$
In this case, conditions are partial functions ordered by extension on each coordinate as well as on the domain,  $p\leq q$ if and only if $\dom q\subseteq \dom p$ and for all $\alpha \in \dom q$, we have $q(\alpha)\subseteq p(\alpha)$.

Finally, it will be convenient in the case $\lambda=\kappa$ to assume conditions have no gaps in their domains, in particular:
$$\Add(\kappa,\kappa) = \left\{p: \alpha \to {}^{<\kappa} 2 \; \mid \; \alpha<\kappa \right\}.$$
\end{definition}

Note that all types of Cohen forcing satisfy the $\left(2^{<\kappa}\right)^\plus$ chain condition and are $<\kappa$-closed, preserving cardinals outside the interval $[\kappa^+,2^{<\kappa}]$, and in fact will collapse cardinals if and only if $2^{<\kappa}>\kappa$ holds in $V$. However, these standard facts need not a priori apply over $V$ to the inner-model versions $\Add(\kappa,1)^W$ and $\Add(\kappa,\lambda)^W$.

\section{When does Cohen forcing from W add a generic for Cohen forcing from V?}

The main theorem of this paper, Theorem \ref{Theorem.WhenDoesCohenForcingFromWAddAGenericForV}, provides a somewhat surprising answer to Question \ref{Question.DoesCohenforcingfromWaddagenericforV}.  Under very natural size constraints (the posets must have the same cardinality), Cohen forcing from $W$ always adds a generic for Cohen forcing from $V$.

\begin{theorem}\label{Theorem.WhenDoesCohenForcingFromWAddAGenericForV}
Given an inner model $W \subset V$ and $\kappa$ a regular cardinal in $V$, forcing over $V$ with $\Add(\kappa,1)^W$ adds a generic filter for $\Add(\kappa,1)^V$ if and only if $\Card{\Add(\kappa,1)^W}=\Card{\Add(\kappa,1)^V}$.
\end{theorem}

\proof For the forward direction, suppose $\Card{\Add(\kappa,1)^W} \neq \Card{\Add(\kappa,1)^V}$.  Then the forcing from $V$ is, of necessity, strictly larger than the forcing from $W$, that is $(2^{<\kappa})^V > (2^{<\kappa})^W \geq \kappa$.  The forcing from $V$ will collapse $(2^{<\kappa})^V$ to $\kappa$, but the forcing from $W$ is too small to collapse this cardinal.

For the reverse direction, I will construct in $V$ a \emph{projection map} from $\Add(\kappa,1)^W$ to $\Add(\kappa,1)^V$.  Projection maps are a standard notion,  dual to the more common set-theoretic construction of complete embeddings, that allow a generic subset of the domain to be pushed forward into a generic subset of the codomain (see \cite{Jech:SetTheory3rdEdition}).

\begin{definition} If $\P$ and $\Q$ are partial orders then a \emph{projection} from $\P$ to $\Q$ is a map $\pi:\P \to \Q$ such that
\begin{itemize}
    \item $\pi(\one_\P)=\one_\Q$
    \item $\pi$ is order-preserving, and
    \item for all $p\in\P$ and all $q\leq \pi(p)$, there is $p^\prime < p$ with $\pi(p^\prime)\leq q$.
\end{itemize}
\end{definition}

The basic fact about projection maps is:

\begin{theorem}If $\pi:\P \to \Q$ is a projection map and $G\subset \P$ is $V$-generic, then $\overline{\pi^{\prime\prime}G}$ (the upwards closure of the image of $G$) is $V$-generic for $\Q$.\qed
\end{theorem}

Given a projection map $\pi:\P \to \Q$, in some cases we can find a complete embedding $i:\Q \to \P$ for which $\pi$ is inverse.  However, the following is always the case:

\begin{theorem}If $\pi:\P \to \Q$ is a projection map then $\pi^{-1}$ generates a complete embedding $i:\Q\to ro(\P)$, where $ro(\P)$ is the algebra of regular open subsets of $\P$ (the Boolean completion of $\P$).
\end{theorem}

\proof Let $i(q)$ be the least regular open subset of $\P$ containing $\pi^{-1} (\{q\})$.\qed

To complete the proof of Theorem \ref{Theorem.WhenDoesCohenForcingFromWAddAGenericForV}, we construct a projection map from a dense subset of $\Add(\kappa,1)^W$ to $\Add(\kappa,1)^V$.  Let $\P$ be the set of all conditions in $\Add(\kappa,1)^W$ whose length is an odd ordinal and whose final digit is $1$.  As any condition in $\Add(\kappa,1)^W$ can be extended to a condition of odd length ending in a $1$, this set is dense in (and therefore forcing equivalent to) $\Add(\kappa,1)^W$.

Under the assumption $\Card{\Add(\kappa,1)^W}=\Card{\Add(\kappa,1)^V}$, fix any bijection $d$ from $\Add(\kappa,1)^W$ to $\Add(\kappa,1)^V$ satisfying $d(\emptyset)=\emptyset$.  While $d$ is not assumed to have any other special properties (e.g. $d$ need not preserve $\leq$), it will be used to define the projection map $\pi:\P \to \Add(\kappa,1)^V$. 

Given $p\in\P$, the value $\pi(p)$ will defined in three steps.  First, we decode $p$ into a sequence $\<p_\alpha \mid \alpha<\gamma>$ of length $\gamma<\kappa$ of conditions in $\Add(\kappa,1)^W$, as follows. Each even digit of $p$ with value $1$ identifies a `split point' of $p$.  The odd digits lying between the $\alpha$th and $(\alpha+1)$th split points give the binary sequence $p_\alpha \in \Add(\kappa,1)^W$.  Note that the length $\gamma$ of the sequence is of necessity $\leq \len(p)<\kappa$.

Next, apply the map $d$ to each member of $\<p_\alpha>$ in turn, yielding a sequence of conditions $\<d(p_\alpha)\mid \alpha<\gamma>$ in $\Add(\kappa,1)^V$. Finally, concatenate these conditions to form a single condition $\pi(p)=d(p_0)^\frown d(p_1)^\frown \dots {}^\frown d(p_\alpha)^\frown \dots$ in $\Add(\kappa,1)^V$.

It remains to show that $\pi$ is a projection map.  As each condition in $\P$ ends with a split point (a value $1$ on an even coordinate), extending a condition $p$ to $p^\prime$ will yield a sequence $\<p_\alpha^\prime>$ that includes $\<p_\alpha>$ as an initial subsequence.   Applying $d$ element-wise to $\<p_\alpha^\prime>$ gives a sequence $\<d(p_\alpha^\prime)>$ which extends $\<d(p_\alpha)>$, and the resulting concatenation gives a condition $\pi(p^\prime)$ extending $\pi(p)$.  Thus $\pi$ preserves ordering.

Now fix any $p\in \P$ and any $q \leq \pi(p)$.  We will find a condition $p^\prime$ extending $p$ such that $\pi(p^\prime) = q$.   Let $z$ be the part of $q$ that lies above $\dom(\pi(p))$, so that $q=\pi(p)^\frown z$.  This is a binary sequence of length $<\kappa$ in $V$, so it lies in $\Add(\kappa,1)^V$.  Therefore we can apply $d^{-1}(z)=x \in \Add(\kappa,1)^W$.  Now extend $p$ to $p^\prime$ by concatenating $p$ with the string consisting of the digits of $x$ on the odd coordinates, and all $0$s on the even coordinates, followed by a $1$ on a final even coordinate.  Clearly $p^\prime < p$, and by construction $\pi(p^\prime) = q$.  Thus $\pi$ is a projection map. \qed

\section{Term forcing for $\Add(\kappa,1)$}

Let us consider Theorem \ref{Theorem.WhenDoesCohenForcingFromWAddAGenericForV} in the event that $V=W[H]$ itself arises as a forcing extension of $W$ by a partial order $\P \in W$.  In this case, Cohen forcing over $V$ at $\kappa$ will yield a forcing extension of $W$, either by the two-step product $\P\times\Add(\kappa,1)^W$ (if the ground model Cohen forcing is used), or by the two-stage iteration $\P*\dot\Add(\kappa,1)$ (if the Cohen forcing of the extension $V$ is used).  Provided $\P$ does not add too many small subsets to $\kappa$, Theorem \ref{Theorem.WhenDoesCohenForcingFromWAddAGenericForV} shows that the extension $W[H\times G]$ by the product contains an extension $W[H* G^\prime]$ by the iteration.  However, the commutative property of products allows us to rearrange the order of the forcing $\P\times\Add(\kappa,1)^W$, forcing with $\Add(\kappa,1)^W$ first and then $\P$, to obtain the final model $W[G][H]$.  
The iteration $\P*\dot\Add(\kappa,1)$ does not share this commutative property, but this observation provides a kind of alternative:  If we want to rearrange the order of the iteration $\P*\dot\Add(\kappa,1)$, we can accomplish it by forcing first with $\Add(\kappa,1)^W$ and then with $\P$.  This yields an extension containing (though not necessarily equal to) an extension by the iteration $\P*\dot\Add(\kappa,1)$, in which the forcing by $\P$ was performed in the final stage rather than the initial stage.  This  behavior suggests a connection with term forcing (or termspace forcing - the nomenclature varies in the literature).  The following account of term forcing is taken from \cite{HamkinsWoodin:NMPccc}.

\begin{definition}Suppose $\P$ is a partial order and $\dot\Q$ is a $\P$-name for a partial order.  The \emph{term forcing $\Q_\term$ for $\dot\Q$ over $\P$} consists of conditions $q$ such that $\forces_\P q\in\dot\Q$, with the order $p \leq_{\Q_\term} q$ if and only if $\forces_\P p \leq_{\dot\Q} q$.
\end{definition}

Note that $\Q_\term$ may be a proper class, as $V^\P$ contains many names for identical elements of $\dot\Q$.  This problem can be avoided by restricting $\Q_\term$ to a full set $B$ of names, so that for any $\P$-name $q$ with $\forces_\P q\in\dot\Q$ there is $p\in B$ with $\forces_\P q=p$.  As such a $B$ forms a dense subset of $\Q_\term$ it is forcing equivalent, and so we will assume without loss of generality that $\Q_\term$ is a set.

The standard fact about term forcing, providing a kind of ``commutativity for iterations,'' is as follows.

\begin{lemma}\label{Lemma.TermForcingBasics}Suppose that $H_\term \subset \Q_\term$ is $V$-generic for the term forcing of $\dot\Q$ over $\P$ and $\Vbar$ is any model of set theory with $V[H_\term] \subset \Vbar$.  If there is a $V$-generic filter $G\subset \P$ in $\Vbar$, then there is a $V[G]$-generic filter $H\subset \Q = \dot\Q_G$ in $\Vbar$.
\end{lemma}

In the case of small forcing followed by Cohen forcing, Theorem \ref{Theorem.CohenForcingIsTermForcing} says that the term forcing is closely related to the ground model Cohen forcing, and in many cases they are forcing equivalent.  Before stating the theorem, we will need an additional fact about Cohen forcing and its variants (our thanks to the referee for abstracting this lemma from an earlier proof of Theorem \ref{Theorem.CohenForcingIsTermForcing}).

\begin{lemma}\label{Lemma.CohenForcingAndVariantsAreEquivalent} Suppose $\kappa$ is regular, $\lambda$ is a cardinal with $\lambda \leq 2^\delta$ for some $\delta<\kappa$.  Let $\P={}^{<\kappa} 2=\Add(\kappa,1)$ and $\Q={}^{<\kappa} \lambda$ (in each case, sequences are ordered by extension).  Then there are projection maps $\pi_0 : \Q \to \P$ and $\pi_1 : \P \to \Q$.  Furthermore, if $\lambda=2^\delta$ then there is a dense embedding from $\P$ to $\Q$ (and hence $\P$ and $\Q$ are forcing equivalent).
\end{lemma}

\proof Fix $\kappa,\lambda$ as in the lemma.  For $\pi_0$, fix any nonconstant function $d_0:\lambda\to 2$.  For $q\in\Q$ and $\beta\in \dom(q)$, let $\pi_0(q)(\beta)=d_0(q(\beta))$, so $\pi_0(q)$ will be a binary sequence of the same length as $q$.  Establishing that $\pi_0$ is a projection map follows from the observation that extending a condition $q$ yields an extension to $\pi_0(q)$, and any extension of $\pi_0(q)$ can be realized as $\pi_0(q^\prime)$ for an appropriate extension $q^\prime$ of $q$.

To construct the map $\pi_1$, we will without loss of generality restrict $\P$ to conditions whose domain is a multiple of $\delta$, as such conditions are dense in $\P$.  Fix any onto function $d_1\from {}^\delta 2 \to \lambda$.  For $p\in\P$ with $\dom(p)=\delta\cdot\alpha$ and $\beta<\alpha$, let $p_\beta$ be the $\beta$th $\delta$-block of $p$, that is, $p_\beta$ is the binary $\delta$-sequence appearing in $p$ on the interval $[\delta\cdot\beta,\delta\cdot(\beta+1))$.  For each such $\beta$, let $\pi_1(p)(\beta)=d_1(p_\beta)$.  Thus $\pi_1(p)$ will be a member of ${}^{<\kappa} \lambda$ of length $\alpha$.  It is a straightforward matter to show that $\pi_1$ is a projection map.

In the event that $\lambda=2^\delta$, we can choose our map $d_1$ to be a bijection.  In this case, the map $\pi_1$ will preserve $\perp$ as well as $\parallel$.  As $\pi_1$ is onto, we conclude that it is a dense embedding and therefore $\P$ and $\Q$ are forcing equivalent. \qed

\begin{theorem}\label{Theorem.CohenForcingIsTermForcing}Suppose $\kappa$ is regular and $\P$ is a partial order with $|\P|<\kappa$.  Let $\dot\Q$ be a $\P$-name for $\Add(\kappa,1)^{V^\P}$, and $\Q_\term$ be the term forcing for $\dot\Q$ over $\P$.
Then there are projection maps $\pi_0 : \Q_\term \to \Add(\kappa,1)$ and $\pi_1:\Add(\kappa,1) \to \Q_\term$ (where $\Add(\kappa,1)$ refers to the Cohen forcing defined in the ground model).  Furthermore, if $\scrA$ is the set of equivalence classes of antichains of $\P$ under the relation ``has a common refinement,'' then $|\scrA|=2^\delta$ for some $\delta<\kappa$ implies that there is a dense embedding from $\Q_\term$ to $\Add(\kappa,1)$ (and hence they are forcing equivalent\footnote{I take forcing equivalence of $\P$ and $\Q$ to mean that every $\P$-forcing extension can be realized as a $\Q$-forcing extension, and vice versa.}).
\end{theorem}

\proof  Fix $\P$, $\kappa$ as in the theorem.  If $A$ is the set of antichains of $\P$, we consider the natural equivalence relation on $A$ given by $a\sim_\P b$ if and only if $a$ and $b$ have a common refinement (in forcing terms, $a\sim_\P b$ if and only if any generic filter $G\subset \P$ meeting one of them necessarily meets the other).  Let $\scrA$ be the set of equivalence classes of $\sim_\P$.  Note that $|\scrA|\leq 2^{|\P|}$ and $|\P|<\kappa$.

Without loss of generality, let us take $\dot\Q$ to consist of nice $\P$-names for bounded subsets of $\kappa$, so $\dot\Q$ can be identified with $\{q : \alpha \to \scrA \mid \alpha < \kappa\}$.  It follows that $\Q_\term$ can be taken of the same form, and in fact that we can take $\Q_\term$ to have the same underlying set as $\dot\Q$ (though the ordering relations may differ). Thus $\Q_\term$ can be identified with ${}^{<\kappa} \scrA$.  Taking $\lambda = |\scrA|$, the conclusion of the theorem follows from  Lemma \ref{Lemma.CohenForcingAndVariantsAreEquivalent}.\qed

We conclude that in the case $|\scrA|=2^\delta$ the term forcing for $\Add(\kappa,1)^{V^\P}$ over small forcing $\P$ \emph{is} the ground model forcing $\Add(\kappa,1)$. It is natural to wonder whether the requirement $|\scrA|=2^\delta$ could be dropped, and still obtain forcing equivalence.  If $|\scrA|$ does not lie in the range of the power function $\alpha \mapsto 2^\alpha$, we still obtain projection maps in both directions between $\Q_\term$ and $\Add(\kappa,1)$.  Thus, given a generic extension $V[G]$ by $\Q_\term$ we can find an intermediate model $V[H]$ where $H=\overline{\pi_0^{\prime\prime}G}$, which is an extension by $\Add(\kappa,1)$.  In $V[H]$ we can obtain yet another intermediate model $V[G^\star]$ where $G^\star = \overline{\pi_1^{\prime\prime}H}$, an extension by $\Q_\term$.  However, defining the maps $\pi_0$ and $\pi_1$ as above will yield $V[G^\star] \subsetneq V[G]$.  It is an open question whether some alternative strategy could demonstrate forcing equivalence in this case.

\section{Extending to $\Add(\kappa,\lambda)$}

\begin{theorem}\label{Theorem.AddKappaLambdaProjectionMap}
If $\kappa$ is a regular cardinal in $V$, $\lambda$ is an ordinal, and $W\subset V$ is an inner model satisfying:
\begin{enumerate}
    \item $|\Add(\kappa,1)^W|=|\Add(\kappa,1)|$, 
    \item $|\Add(\kappa,\lambda)^W|=|\Add(\kappa,\lambda)|$, and 
    \item $W\subset V$ satisfies the $\kappa$-cover property for subsets of $\lambda$ \emph{(that is, any $A\subset \lambda$, $A\in V$ with $|A|<\kappa$ is contained in some $B \in W$ with $|B|^W<\kappa$)}.
\end{enumerate}
then there is a projection map $\pi$ from a dense subset of $\Add(\kappa,\lambda)^W$ to $\Add(\kappa,\lambda)$.
\end{theorem}

\begin{proof}
Recall that conditions $p\in \Add(\kappa,\lambda)$ are partial functions $p\from \lambda \to {}^{<\kappa}2$ with $|p|<\kappa$.   Note that any such function is uniquely determined by the domain $A \subset \lambda$ (where $A$ has order type $\gamma<\kappa$) and a member of $\Add(\kappa,\kappa)$, that is, a sequence $\<p_\alpha \in {}^{<\kappa}2 \mid \alpha < \gamma>$ giving the value of $p$ on the $\alpha^\text{th}$ member of the domain $A$.  

To produce our projection map we will begin with a map $d:\Add(\kappa,1)^W \to \Add(\kappa,1)$.  In contrast to the proof of Theorem \ref{Theorem.WhenDoesCohenForcingFromWAddAGenericForV}, however, $d$ will not be a bijection but a many-to-one map -- the intention is to provide a great deal of latitude in choosing preimages for conditions in $\Add(\kappa,1)$.  We construct $d$ as follows:

Let $\<a_\alpha> \in W$ and $\<b_\alpha>$ be sequences of the same length satisfying:
\begin{itemize}
    \item $\<a_\alpha>$ enumerates the conditions in $\Add(\kappa,1)^W$,
    \item each $b_\alpha \in \Add(\kappa,1)$, 
    \item $a_0 = b_0 = \emptyset$, and
    \item every member of $\Add(\kappa,\kappa)$ appears as a block in $\<b_\alpha>$ -- that is, for any $p:\gamma \to {}^{<\kappa}2$ with $\gamma<\kappa$ there is $\alpha_p$ such that $b_{\alpha_p + \beta} = p(\beta)$ for every $\beta<\gamma$.
\end{itemize}
Note that such sequences exist by hypothesis, as $|\Add(\kappa,1)^W|=|\Add(\kappa,1)|=|\Add(\kappa,\kappa)|$.  Define $d(a_\alpha)=b_\alpha$ for all $\alpha$.

Let $\P\subseteq \Add(\kappa,\lambda)^W$ be the set of all conditions $p$ satisfying for all $\beta\in \dom(p)$, $p(\beta)$ is a member of $(2^{<\kappa})^W$ of odd ordinal length whose final digit is $1$.  $\P$ is dense in $\Add(\kappa,\lambda)^W$, as any individual $p(\beta)$ can easily be extended to such a sequence.  To define $\pi:\P \to \Add(\kappa,\lambda)$ we apply the strategy from the proof of Theorem \ref{Theorem.WhenDoesCohenForcingFromWAddAGenericForV} coordinate-by-coordinate.  For $p\in \P$ and $\beta \in \dom(p)$, we construct $\pi(p)(\beta)$ by first decoding $p(\beta)$ into a sequence $\<p(\beta)_\alpha>$ of conditions in $\Add(\kappa,1)^W$ (using the even digits with value 1 to determine split points in the sequence of odd digits), applying the map $d$ to each $p(
\beta)_\alpha$ in turn, and concatenating the results to obtain $\pi(p)(\beta)=d(p(\beta)_0)^\frown d(p(\beta)_1)^\frown \dots {}^\frown d(p(\beta)_\alpha)^\frown \dots \in \Add(\kappa,1)$.   Note that $\pi(p)$ will have the same domain as $p$.

To see that $\pi$ is a projection map, note first that $\pi$ is order-preserving, as it is order-preserving on each coordinate.  Now fix $p\in \P$ and $q\leq \pi(p)$.  We will construct $p^\prime \leq p$ with $\pi(p^\prime)\leq q$.  First, by hypothesis (3) (cover property) we can extend $q$ to $q^\prime$ with $\dom(q^\prime) \in W$ (without loss of generality we can also assume that $q^\prime(\beta)$ properly extends $\pi(p)(\beta)$ for all $\beta\in\dom(\pi(p))$.  Now let $z$ be the difference between $\pi(p)$ and $q^\prime$ -- that is, for each $\beta \in \dom(q^\prime)$, let $z(\beta) \in {}^{<\kappa}2$ be the unique sequence such that $q^\prime(\beta)=\pi(p)(\beta)^\frown z(\beta)$.  Here we would like to use the map $d$ to pull $z$ coordinate-wise back into a condition in $\Add(\kappa,\lambda)^W$, but we must take a little care -- although we can easily select a $d$-preimage of each $z(\beta)$, it is not clear that the collection of these preimages will be a set in $W$.  However, recall that $z$ is determined by its domain together with a member of $\Add(\kappa,\kappa)$, and every member of $\Add(\kappa,\kappa)$ appears as a block in the $\<b_\alpha>$ sequence.  We simply take $x\in \Add(\kappa,\kappa)^W$ to be the corresponding block in the $\<a_\alpha>$ sequence.  This block lies in $W$ and from it, together with $\dom(z)$ (which lies in $W$ by construction), we can define a condition $x^\prime \in \Add(\kappa,\lambda)^W$ that maps coordinate-wise under $d$ to $z$.  We now construct $p^\prime \leq p$:  for each $\beta \in \dom(z)$, we extend $p(\beta)$ to $p^\prime(\beta)$ by concatenating $p(\beta)$ with the string consisting of the digits of $x^\prime(\beta)$ on the odd coordinates, and all $0$s on the even coordinates, followed by a $1$ on a final even coordinate.  Clearly $p^\prime < p$, and by construction $\pi(p^\prime) = q^\prime \leq q$.
\end{proof}


\section{Generalized Cohen Iterations}

We now consider the repeated application of Cohen forcing at distinct cardinals, which arises in many  applications and generally appears in one of two basic forms, products and iterations.
While products can be considered a special case of iterations (in which the forcing at stage $\alpha$ always consists of a poset from the ground model), they are often treated separately.  In the case of proper class products and iterations, $\ZFC$-preservation requires additional assumptions - in many `nice' cases, these take the form of factoring assumptions in which the tail forcing satisfies a progressive closure condition (details below) - but the precise formulation of these assumptions varies between products and iterations.  We will use the ideas from previous sections to analyze a broad new class of forcing notions, iterations in which each stage consists of Cohen forcing taken from an arbitrary model intermediate between the ground model and the extension.  These lack the progressive closure conditions of typical products and iterations, but nonetheless preserve $\ZFC$.  We  begin with a brief discussion of products and iterations.

Easton's celebrated theorem on controlling the continuum function $\kappa \mapsto 2^\kappa$ for regular cardinals $\kappa$ provides a canonical example of a class product of Cohen partial orders.  Given $\GCH$ in $V$, a class $I$ of regular cardinals, and an appropriate function $\kappa \mapsto \lambda_\kappa$ with domain $I$, the Easton product $\Pi_{\kappa\in I} \Add(\kappa,\lambda_\kappa)$ allows tight control over the continuum function on $I$ while preserving $\ZFC$.  Preservation of $\ZFC$, especially Power Set and Replacement, are proved using a factoring property for products, namely, that for arbitrarily large $\kappa$ the product factors as $\P\times\Q$ where $\P$ is set forcing with the $\kappa^+$-c.c. and $\Q$ is $\leq\kappa$-closed (in the ground model $V$).  In \cite{Reitz2006:Dissertation} I refer to such forcings as \emph{progressively closed products}, and crucially it is the $V$-closure of the second factor that allows us to prove preservation of Power Set and Replacement.

The complementary theorem, forcing the $\GCH$ to hold in the extension regardless of the behavior of the continuum function in $V$, provides a similar canonical example for iterations of Cohen forcing.  This is achieved via the Easton support iteration $\<\P_\kappa, \dot\Q_\kappa \mid \kappa \in \REG>$ in which we force at stage $\kappa$ to add a single Cohen subset using the partial order as defined in the extension up to $\kappa$, that is, $\P_\kappa \forces \dot\Q_\kappa = \dot\Add(\check\kappa,1)$.  While this forcing may collapse cardinals, it will nonetheless preserve $\ZFC$ in the extension, and once again it is a factoring property that allows us to show preservation of Power Set and Replacement.  Here, we have that for arbitrarily large $\kappa$ we can factor the iteration as $\P_\kappa * \P_\text{tail}$ where $\P_\kappa \forces \P_\text{tail} \text{ is }<\kappa\text{-closed}$.  These are termed \emph{progressively closed iterations} in \cite{Reitz2006:Dissertation}, and though similar to the case of progressively closed products, they differ in that the tail forcing satisfies a closure property in the extension  $V^{\P_\kappa}$, rather than in the ground model $V$.

However, not every repetition of Cohen forcing is a progressively closed product or iteration.  For example, suppose we divide the regular cardinals into $\ORD$-many disjoint classes $\REG = \bigsqcup_{\alpha \in \ORD} I_\alpha$ and for each class $I_\alpha$ we consider the progressively closed iteration adding a Cohen subset to each cardinal in $I_\alpha$.  We then take the class product of these iterations for all $\alpha\in\ORD$.  It is natural to view this class-product-of-class-iterations as a single class iteration of set forcing along the regular cardinals, forcing at stage $\kappa$ with the poset $\Add(\kappa,1)$ as defined in the inner model obtained by including generics only at stages in $I_\alpha \cap \kappa$, where $I_\alpha$ is the unique class containing $\kappa$.  Unfortunately, from this perspective the forcing is not 
a progressively closed iteration (the forcing at stage $\kappa$ may no longer be $<\kappa$-closed, since there may be bounded subsets of $\kappa$ added by earlier stages not in $I_\alpha\cap\kappa$). Thus, the analysis of this forcing (factoring arguments, preservation of $\ZFC$ and of cardinals, and so on) requires a new approach.

In this section we consider iterations of Cohen forcing for which, at each stage $\kappa$, we force with the poset $\Add(\kappa,\lambda)$ as defined over some intermediate inner model.

\begin{definition}\label{Definition.GeneralizedCohenIteration}Suppose $\P = \<\P_\kappa, \dot\Q_\kappa \mid \kappa \in I>$ is an iteration along a class $I$ of regular cardinals with either full (set) support or Easton support.  Then $\P$ is a \emph{generalized Cohen iteration} provided, for each $\kappa \in I$,

\begin{enumerate}
    \item $\dot\Q_\kappa$ is a full $\P_\kappa$-name for a partial order and $\P_\kappa \forces \dot\Q_\kappa = \dot\Add(\kappa, \lambda_\kappa)^{V^{\R_\kappa}}$, where
    \item $\R_\kappa$ is a (possibly trivial) complete suborder of $\P_\kappa$
   \item $\displaystyle \R_\kappa \forces \left|\check\Add(\check\kappa,1)\right| = \left|\dot\Add(\kappa, 1)^{V^{\R_\kappa}}\right|$, \\ 
    $\displaystyle \R_\kappa \forces \left|\check\Add(\kappa,\lambda_\kappa)\right| = \left|\dot\Add(\kappa, \lambda_\kappa)^{V^{\R_\kappa}}\right|$, and \\
    $\displaystyle \R_\kappa \forces \check{V}\subseteq V^{\R_\kappa}$ satisfies the $\kappa$-cover property for subsets of $\lambda_\kappa$.
\end{enumerate}
\end{definition}

\begin{theorem}\label{Theorem.GeneralizedCohenIterations}Suppose $\P$ is a generalized Cohen iteration. 
Then there is a projection map $\pi:\Pi_{\kappa\in I} \Add(\kappa,\lambda_\kappa)\to\P$, where  $\Pi_{\kappa\in I} \Add(\kappa,\lambda_\kappa)$ is the ground model class product (with the same support as $\P$). Furthermore, 
\begin{enumerate}
    \item $\pi$ restricted to an initial segment of the product maps to the corresponding initial segment of $\P$, that is 
    $$\pi_\delta : \Pi_{\kappa \in I\cap\delta} \Add(\kappa,\lambda_\kappa) \to \P_\delta=\left<\P_\kappa, \dot\Q_\kappa \mid \kappa \in I\cap\delta\right>$$ is a projection map, and for $p\in\P$ we have $\pi(p\restrict\delta)=\pi(p)\restrict\delta$
    \item $\P$ is a \emph{progressively distributive iteration}, i.e. for each $\delta \in I$ we can factor $\P \cong \P_\delta * \P_\text{tail}$ where $\P_\text{tail}$ is a $\P_\delta$-name for the tail forcing and $\P_\delta \forces \P_\text{tail}$ is $<\delta$-distributive, and thus
    \item  $\P$ preserves \ZFC.
\end{enumerate}
\end{theorem}

\proof
The projection map $\pi$ is constructed level-by-level, using the results from the previous section.  
Suppose we have completed the construction up to some $\delta$, so we have a projection map $\pi_\delta : \Pi_{\kappa \in I\cap\delta} \Add(\kappa,\lambda_\kappa) \to \left<\P_\kappa, \dot\Q_\kappa \mid \kappa \in I\cap\delta\right>$.  To extend $\pi_\delta$ to conditions with domain $\subseteq \delta+1$, we consider the following sequence of models:
$$V\subseteq V^{\R_\delta} \subseteq V^{\P_\delta} \subseteq V^{\Pi_{\kappa \in I\cap\delta}\Add(\kappa,\lambda_\kappa)}$$
 In each case the containment comes from either a projection map or a complete embedding.  We will show that forcing with the ground model poset $\Add(\delta,\lambda_\delta)$ over the largest model $V^{\Pi_{\kappa \in I\cap\delta}\Add(\kappa,\lambda_\kappa)}$ will add a generic for $\dot\Q_\delta = \dot\Add(\delta, \lambda_\delta)^{V^{\R_\delta}}$.  To accomplish this, we first obtain a projection map $f:\Add(\delta, \lambda_\delta)^V \to \dot\Add(\delta, \lambda_\delta)^{V^{\R_\delta}}$ in the model $V^{\R_\delta}$, and then argue that $f$ remains a projection map in $V^{\Pi_{\kappa \in I\cap\delta}\Add(\kappa,\lambda_\kappa)}$.

To obtain $f$, we apply Theorem \ref{Theorem.AddKappaLambdaProjectionMap} to the partial order $\dot{\Q_\delta}=\dot\Add(\delta, \lambda_\delta)^{V^{\R_\delta}}$ in the model $V^{\R_\delta}$ with regards to the inner model $V \subseteq V^{\R_\delta}$ - the hypotheses of the theorem are satisfied in this model exactly by the cardinality and cover property assumptions of Definition \ref{Definition.GeneralizedCohenIteration}.(3).  Thus there is, in $V^{\R_\delta}$, a projection map $f:\Add(\delta,\lambda_\delta)^V\to\dot{\Q_\delta}$.  As being a projection map depends only on $\Delta_0$ properties of the function, domain, and codomain, it follows that $f$ remains a projection map in the larger model $V^{\Pi_{\kappa \in I\cap\delta}\Add(\kappa,\lambda_\kappa)}$.  Fix a $\Pi_{\kappa \in I\cap\delta}\Add(\kappa,\lambda_\kappa)$-name $\dot{f}$ for $f$ such that $\one_{\Pi_{\kappa \in I\cap\delta}\Add(\kappa,\lambda_\kappa)} \forces \dot{f} \from \Add(\delta,\lambda_\delta) \to \dot{\Q_\delta}$ is a projection map.\footnote{This statement is not strictly correct as written, since the target $\dot{\Q_\delta}$ is not a $\Pi_{\kappa \in I\cap\delta}\Add(\kappa,\lambda_\kappa)$-name but rather a $\P_\delta$-name for an object in $V^{\R_\delta}$.  However, it follows from the properties of the projection map $\pi_\delta$ that every object $a \in V^{\P_\delta}$ has a name $a^\star\in V^{\Pi_{\kappa \in I\cap\delta}\Add(\kappa,\lambda_\kappa)}$, so that if $G\subseteq \Pi_{\kappa \in I\cap\delta}\Add(\kappa,\lambda_\kappa)$ is generic then valuating $a^\star$ by $G$ yields the same result as valuating $a$ by $\overline{\pi_\delta^{\prime\prime}G}$.  In the interest of clarity I will suppress the ${}^\star$ notation and treat $a$ and $a^\star$ as the same object, denoted $a$.}

Now fix $p\in \P_{\delta+1}$.  Let $\pi_{\delta+1}(p) = \pi_\delta(p\restrict \delta)^\frown \dot q$, where $\dot q\in \dom \dot\Q_\delta$ is selected so that $\one_{\Pi_{\kappa \in I\cap\delta}\Add(\kappa,\lambda_\kappa)} \forces \dot{f}(p(\delta))=\dot q$ (fullness of the name $\dot\Q_\delta$ guarantees that such a $\dot q$ will exist in $\dom(\dot\Q_\delta)$).

That $\pi_{\delta+1}$ is a projection map follows from the fact that $\dot{f}$ is forced to be projection map in the extension by  $\Pi_{\kappa \in I\cap\delta}\Add(\kappa,\lambda_\kappa)$.

For the limit case, suppose we have completed the construction up to some limit $\delta$, so for each $\alpha<\delta$ we have a projection map $\pi_\alpha : \Pi_{\kappa \in I\cap\alpha} \Add(\kappa,\lambda_\kappa) \to \P_\alpha$, and furthermore the $\pi_\alpha$ cohere in the sense of hypothesis (1).  Then for $p\in\P_\delta$, we can define $\pi_\delta(p) = \bigcup_{\alpha<\delta} \pi_\alpha (p\restrict \alpha)$.  It is a straightforward matter to verify that $\pi_\delta$ is a projection map as desired.  Thus the construction can be carried out inductively through the ordinals, yielding a a projection map $\pi\from\Pi_{\kappa \in I}\Add(\kappa,\lambda_\kappa)\to \P$ satisfying condition (1) of the theorem.  

To establish progressive distributivity, we fix $\delta\in I$ and factor $\P=\P_\delta*\P_\text{tail}$.  To see that $\P_\delta \forces \P_\text{tail}$ is $<\delta$-distributive, we factor the product as $\Pi_{\kappa \in I}\Add(\kappa,\lambda_\kappa) \cong \Pi_{\kappa \in I\cap\delta}\Add(\kappa,\lambda_\kappa) \times \Pi_{\kappa \in I\setminus \delta}\Add(\kappa,\lambda_\kappa)$ and observe that a counterexample to distributivity for $\P_\text{tail}$ can be pulled back via the $\pi$ map to yield a counterexample to distributivity for $\Pi_{\kappa \in I\setminus \delta}\Add(\kappa,\lambda_\kappa)$.  It is a well-known fact from the theory of product forcing that the latter product is $<\delta$-distributive, and so such a counterexample cannot exist.

Preservation of $\ZFC$ follows from the progressive distributivity condition, requiring only minor adjustments to the corresponding proof for progressively closed iterations \cite{Reitz2006:Dissertation} (details of the necessary adjustments appear in \cite{ReitzWilliams:InnerMantlesandIteratedHOD}.) 
\qed

The strategy above can be generalized to a deal with a variety of iterations that are constructed from Cohen `building blocks' but do not strictly satisfy definition \ref{Definition.GeneralizedCohenIteration}.  For example, an iteration that forces at each stage with the lottery sum of several Cohen posets (allowing the generic to choose, for example, whether to force $\GCH$ or its negation at $\kappa$) can be similarly embedded into a ground model product by replacing each term in the lottery sum with the corresponding Cohen forcing from the ground model -- provided we can verify that conditions (1)-(3) of definition \ref{Definition.GeneralizedCohenIteration} hold for each term in the lottery at each stage.  Several examples of this type can be found in the proofs of the main theorems of \cite{ReitzWilliams:InnerMantlesandIteratedHOD}.

\section{When does Cohen forcing from V add a generic for Cohen forcing from W?}\label{Section.WhendoesCohenforcingfromVaddagenericforW}

We turn our attention back to Question \ref{Question.DoesCohenforcingfromVaddagenericforW}, focusing on $\Add(\kappa,1)^V$.  Under what conditions does this forcing add a generic for $\Add(\kappa,1)^W$?  Here, the situation is less clear and the results are mostly negative.  Let us begin with a condition which ensures that the Cohen forcing from $V$ does \emph{not} add a generic for Cohen forcing from $W$.  Hamkins isolated the following useful relationship between models of set theory:

\begin{definition}\cite{Hamkins2003:ExtensionsWithApproximationAndCoverProperties} Suppose
that $W\subseteq V$ are both transitive models of (a
suitable small fragment of) \ZFC\ and $\delta$ is a
cardinal in $V$. Then the extension $W\of V$ exhibits the \emph{$\delta$-approximation property} if whenever $A\in
     V$ with $A\subseteq W$ and $A\cap B\in W$ for all
     $B\in W$ with $|B|^{W} < \delta$, then $A\in W$.
\end{definition}


A variety of common extensions satisfy this property, including a great many forcing extensions.


\begin{lemma}{\cite[lemma~13]{Hamkins2003:ExtensionsWithApproximationAndCoverProperties}}\label{Lemma.ClosurePointForcing}
If $W\subset W[H]$ is a forcing extension by forcing of
the form $\Q_1*\dot\Q_2$, where $\Q_1$ is nontrivial and
$\forces_{\Q_1}$``\hspace{2 pt}$\dot\Q_2$ is ${\leq}|\Q_1|$
strategically closed'', then $W\of W[H]$ satisfies the
$\delta$-approximation property for any
$\delta\geq|\Q_1|^{+}$.
\end{lemma}

We say that such forcing has a closure point at $|\Q_1|$.

For the purpose of addressing Question \ref{Question.DoesCohenforcingfromVaddagenericforW}, note that a generic $G\subset\kappa$ for $\Add(\kappa,1)^W$ serves as a counterexample to the $\kappa$-approximation property (such a $G$ will be a subset of $\kappa$ all of whose initial segments, and hence all of whose small approximations, lie in $W$).  We make the following observation.

\begin{observation}[A sufficient condition]Suppose $W\subset V$ are transitive models of \ZFC, $\kappa$ is regular, and $G\subset \kappa$ is generic for $\Add(\kappa,1)^V$.  If $W\subset V[G]$ satisfies the $\kappa$-approximation property, then the forcing did not add a generic for $\Add(\kappa,1)^W$.
\end{observation}

This will happen, for example, when the extension $W\subset V$ is a forcing extension by forcing which has a closure point below $P$.

\begin{lemma}\label{Lemma.AfterclosurepointsforcingCohenforcingfromVdoesntaddagenericforW}If $W \subset V$ is an extension by forcing with a closure point below $\kappa$, then forcing with $\Add(\kappa,1)^V$ does not add a generic for $\Add(\kappa,1)^W$.
\end{lemma}

\begin{proof}If $W \subset V$ is an extension by forcing $\Q_1*\dot \Q_2$ with a closure point below $\kappa$, then the forcing $\Q_1 * (\dot \Q_2*\dot\Add(\kappa,1))$ also has a closure point below $\kappa$, since $\Add(\kappa,1)$ is $<\kappa$ closed in $V$ (and so $\forces_{\Q_1}$``\hspace{2 pt}$\dot\Q_2*\dot\Add(\kappa,1)$ is ${\leq}|\Q_1|$ strategically closed'').  Note that in the definition of closure point forcing, the second factor $\dot \Q_2$ is allowed to be trivial, which shows that the lemma holds for any extension $W\subset V$ by small forcing (below the size of $\kappa$).
\end{proof}

This result leaves open the question of extensions $W\subset V$ satisfying the $\kappa$-approximation property that do not arise via closure point forcing (such extensions are also common, see \cite{Hamkins2003:ExtensionsWithApproximationAndCoverProperties}).  In particular, I suspect the following:

\begin{conjecture}If $W \subset V$ satisfies the $\kappa$-cover and approximation properties, then forcing with $\Add(\kappa,1)^V$ does not add a generic for $\Add(\kappa,1)^W$.
\end{conjecture}

No example has been found in which forcing with $\Add(\kappa,1)^V$ adds a generic for $\Add(\kappa,1)^W$ (except when the posets are equal), but it is not clear whether this holds generally.

\begin{question}Is it consistent that $\Add(\kappa,1)^V \neq \Add(\kappa,1)^W$ and forcing with $\Add(\kappa,1)^V$  adds a generic for $\Add(\kappa,1)^W$?
\end{question}


\section{Conclusion}

There are several natural applications. Recent joint work with Kameryn Williams on the Inner Mantle sequence and Iterated $\HOD$ sequence \cite{ReitzWilliams:InnerMantlesandIteratedHOD} relies on the notion of generalized Cohen iterations to establish $\ZFC$-preservation.  Another possible application is a robust coding method, not disturbed by small forcing, that is flexible and compatible with any $GCH$ pattern.

In addition, the results and techniques are flexible, and should be readily adaptable to other types of forcing from inner models (e.g. Collapse forcing, etc.).  In general we might consider a forcing notion $\P$ defined by $\phi(x,\bar a)$, and an inner model $W\subset V$ with $\bar a \in W$, and compare the effects of forcing with $\P$ versus the poset $\P^W$ defined by $\phi^W(x,\bar a)$.

\begin{question}\label{Question.DoesPforcingfromWaddagenericforV}
Does forcing with $\P^W$ add a generic for $\P$?
\end{question}

In particular, does the ``upwards propagation'' phenomenon established by Theorem \ref{Theorem.WhenDoesCohenForcingFromWAddAGenericForV} hold for other types of forcing?

\begin{question}\label{Question.DoesPforcingfromVaddagenericforW}Does forcing with $\P$ add a generic for $\P^W$?
\end{question}

Do we see the same ``downwards non-propagation'' phenomenon described in Section  \ref{Section.WhendoesCohenforcingfromVaddagenericforW}?



\printbibliography
\end{document}